\documentclass[reqno,12pt]{amsart}
\usepackage[english]{babel}
\usepackage[utf8]{inputenc}
\usepackage{amsmath}
\usepackage{amsthm}
\usepackage{amssymb}

\usepackage{comment}
\usepackage{enumerate}
\usepackage[dvips]{color}
\usepackage[hidelinks = true]{hyperref}
\usepackage{doi}
\usepackage{cleveref}
\usepackage{xcolor}

\usepackage{graphicx}
\graphicspath{{Immagini/}}
\newtheorem{theorem}{Theorem}[section]
\newtheorem{definition}[theorem]{Definition}

\newtheorem{lemma}[theorem]{Lemma}
\newtheorem{proposition}[theorem]{Proposition}
\newtheorem{corollary}[theorem]{Corollary}

\theoremstyle{definition}
\newtheorem{remark}[theorem]{Remark}

\title[]
{On the generating functions and special functions\\ associated with superoscillations}
 \oddsidemargin
0.2in \evensidemargin 0.2in \topmargin -0.5in \textwidth
15.5truecm \textheight 23truecm

\author[F. Colombo]{F. Colombo}
\address{(FC) Politecnico di
Milano\\Dipartimento di Matematica\\Via E. Bonardi, 9\\20133 Milano\\Italy}
\email{fabrizio.colombo@polimi.it}

\author[R.S. Kraußhar]{R.S. Kraußhar}
\address{(RSK) Chair of Mathematics\\
University of Erfurt\\ Nordh\"auser Stra{\ss}e 63\\ 99089 Erfurt \\Germany }
\email{soeren.krausshar@uni-erfurt.de }

\author[I. Sabadini]{I. Sabadini}
\address{(IS) Politecnico di
Milano\\Dipartimento di Matematica\\Via E. Bonardi, 9\\20133 Milano\\Italy}
\email{irene.sabadini@polimi.it}

\author[Y. Simsek]{Y. Simsek}
\address{(YS) Faculty of Science Department of Mathematics\\ Akdeniz University\\ TR-07058 Antalya\\ Turkey }
\email{ysimsek@akdeniz.edu.tr}

\begin{document}
\maketitle

\begin{abstract}
The aim of this paper is to study generating functions
for the coefficients of the classical superoscillatory function associated with weak measurements.  We also establish some new relations between the superoscillatory coefficients and many well-known families of special polynomials,
numbers, and functions such as Bernstein basis functions, the Hermite polynomials, the Stirling numbers of second kind,
and also the confluent hypergeometric functions.
Moreover, by using generating functions, we are able to develop a recurrence relation and a derivative formula for the superoscillatory coefficients.
\end{abstract}
\vskip 1cm
\par\noindent
 AMS Classification: 33C15, 32A15, 47B38.
\par\noindent
\noindent {\em Key words}:
Key words: Superoscillating functions, generating functions for superoscillations, Stirling numbers, Bernstein basis functions, Hermite polynomials.
\vskip 1cm

\date{today}
\tableofcontents

\section{Introduction}

The Aharonov-Berry superoscillating functions are band-limited functions
that can oscillate faster than their fastest Fourier component.
These functions, more precisely speaking, these sequences appear in the study of weak
measurements (see \cite{aav,av,b5}).
A weak measurement of a quantum observable, represented by a self-adjoint operator $A$,
involving a pre-selected state $\psi_0$ and a post-selected state $\psi_1$, by definition
is given by
$$
A_{weak}:=\frac{(\psi_1,A\psi_0)}{(\psi_1,\psi_0)}=b+ib'.
$$
Clearly $A_{weak}$ is in general a complex number, its real part $b$ and its imaginary part $b'$  can be interpreted
as the shift $b$ and the momentum $b'$ of the pointer recording this measurement.

An important feature of the weak measurement is that, in contrast to the strong measurements of von Neumann
 (given by the expectation value of the operator $A$)
$$
A_{strong}:=(\psi, A\psi),
$$
the real part $b$ of $A_{weak}$ can be very large, because $(\psi_1,\psi_0)$ can be very small when the states $\psi_0$ and $\psi_1$
are almost orthogonal and this is what produces the superoscillations.
The literature related to superoscillations is very abundant, and without
claiming completeness, we mention for example
\cite{berry,berry-noise-2013,berry2,b1,b4}.

\medskip
Quite recently, this class of functions has been investigated from the
mathematical point of view, see for example \cite{acsst1,acsst6,step,QS1,QS3,Alpay1,Alpay2,QS2} and also \cite{ABCS1,JEVOL,JDE,AShushi,tre,CSSY,kempfQS,PETER}
and the monograph \cite{acsst5}. Their theory is now very well developed, even though there are still open problems associated with them, in particular concerning their longevity, when they are evolved according to a wide class of differential equations.

A superoscillatory sequence frequently considered in the context of weak
measurements is of the type
\begin{equation}
F_{n}(x,a)=\left( \cos \left( \frac{x}{n}\right) +ia\sin \left( \frac{x}{n}%
\right) \right) ^{n}=\sum_{k=0}^{n}c_{k}(n,a)e^{i(1-2k/n)x},  \label{FNEXP}
\end{equation}%
where $x\in \mathbb{R}$ and, for $a>1$, the coefficients $c_{k}(n,a)$ are
given by
\begin{equation}
c_{k}(n,a)=\left(
\begin{array}{c}
	n  \\
	k
	
\end{array} \right)\left( \frac{1+a}{2}\right) ^{n-k}\left( \frac{1-a}{2}%
\right) ^{k}.  \label{Ckna}
\end{equation}%
If we fix $x\in \mathbb{R}$ and we let $n$ go to infinity, then we obtain%
\[
\lim_{n\rightarrow \infty }F_{n}(x,a)=e^{iax}.
\]%
In the sequel we write $F_{n}(x)$ instead of $F_{n}(x,a)
$, when it is not important to specify the dependence on the parameter $a$.
Observe that the terms $(1-2k/n)$ that appear in the Fourier representation
of $F_{n}$ are bounded in modulus by one, but the limit function $e^{iax}$
oscillates with frequency $a$ which can be arbitrarily large and this is in
fact the superoscillatory phenomenon. In \cite{BerryMILAN}  M.V.
Berry presents explicit examples of superoscillatory functions that are
different from (\ref{FNEXP}). An important physical problem is to study the
evolution of superoscillatory functions as initial data in the
Schr\"{o}dinger equation and to ask whether the superoscillations persist in time. To
state the problem we need to formalize what we mean by superoscillations
when the functions $e^{i(1-2k/n)x}$ in the expression (\ref{FNEXP}) are not
necessarily exponential. In this case we generalize the notion of
superoscillations with the notion of super-shift, so that superoscillations
become a particular case of supershift.

\medskip
Precisely speaking,  let $I\subseteq \mathbb{R}$ be an interval
and suppose that $x\in \mathbb{R}$. We call \emph{generalized Fourier sequence} a
sequence of the form
\begin{equation}
Y_{n}(x,a):=\sum_{j=0}^{n}E_{j}(n,a)e^{ik_{j}(n)x},  \label{basic_sequence}
\end{equation}%
where $a\in I$, $(k_{j}(n))_{j,n\in \mathbb{N}_{0}}\in I$ and $%
(E_{j}(n,a))_{j,n\in \mathbb{N}_{0}}\in \mathbb{C}$ for all $j=0,...,n$ and $%
n\in \mathbb{N}$.

 A generalized Fourier sequence $Y_{n}(x,a)$ is called \emph{a superoscillating sequence} if:

 (I) The sequence $(k_j(n))_{j,n\in \mathbb{N}_0}$ satisfies the boundedness condition
$|k_j(n)|<1$ for all $j=0,...,n$ and $n\in \mathbb{N}$.

 (II) There exists a compact subset of $\mathbb{R}$, which will be
called \emph{a superoscillation set}, on which
\[
\lim_{n\to\infty}Y_n(x,a)=e^{ig(a)x}
\]
where the convergence is uniform for $x$ in the compact set and $g:I\mapsto
\mathbb{R}$ is a continuous function such that $|g(a)|>1$.

\medskip
The aim of this paper is to study the
 generating functions associated with the
sequence $c_{k}(n,a)$ defined in (\ref{Ckna}).

Given a sequence of numbers $(a_{n})_{n\in \mathbb{N}_{0}}$, where $\mathbb{N%
}_{0}=\mathbb{N}\cup \{0\}$, we define a generating function the formal
power series by
\[
f(t)=\sum_{n=0}^{\infty }a_{n}t^{n}.
\]%
We call this series the {\em generating function of the sequence} $(a_{n})_{n\in
\mathbb{N}_{0}}$. In general, it is not required that this formal power
series is convergent.

There are various types of generating functions,
including for instance exponential generating functions,
Lambert series, Bell series, and Dirichlet series. Generating functions are
often expressed in closed form, namely by some expression involving operations
defined for formal series. The exponential generating function of a sequence
is defined as follows
\[
EG(a_{n};t)=\sum_{n=0}^{\infty }a_{n}{\frac{t^{n}}{n!}}.
\]
Provided that $EG(a_n;t)$ is analytic in an open set containing the origin, the sequence $a_n$ can be explicitly obtained as the sequence of the derivatives with respect to $t$ of the function $EG(a_n,\cdot)$ evaluated at the origin.

The mathematical tools developed in this paper provide the ground for links with mathematical and statistical physics where it is well known the importance of the knowledge of the generating functions solutions to specific differential problems.

\medskip
{\em The plan of the paper and the main results.}

In Section \ref {GEN} we introduce the generating functions for superoscillations. The functions $S_1$ and $S_2$, see Definition \ref{DEFS1S2}, turn out to have deep connections with the coefficients in (\ref{Ckna}) of the superoscillatory function $F_n$.

In particular, the functions $S_1$ and $S_2$ depend on some parameters
$m$, $k$, $n$ and on the auxiliary variables
$(\alpha _{0},\alpha _{1},\ldots ,\alpha _{m})$. For particular values on $m$ we have remarkable relations with the coefficients in (\ref{Ckna}), in particular explicit relations that involve the Stirling numbers of second kind which play a crucial role in the theory of partitions.
 Our treatment using generating functions,
  also provides a significant extension and generalization of a very recently obtain result in \cite{PW} (Lemma 2.1)
   where one particular relation between the coefficients $c_k(n,a)$ and Stirling numbers has been presented.
 We now understand much more profoundly why Stirling numbers of the second kind play a key role in superoscillatory functions.

In Section \ref{recurrence} we study a recurrence relation for the coefficients
$c_{k}(n,x)$ and their relation to special polynomials. The relations between $c_k(n,x)$ and the Bernstein basis functions $B_k^{v}(x)$
and their relation with the Hermite polynomials are of particular interest.
Finally, in Section \ref{CONC} we conclude with the discussion on further developments.

\section{Generating functions for superoscillations} \label{GEN}

Integration and differentiation of generating functions may provide a differential equation and a recurrence relation regarding the polynomial sequence corresponding to the coefficients of the associated generating functions. Therefore, these functions are often used for example in combinatorics, analytic number theory, probability theory and computer science, as well as in mathematical physics, quantum physics and statistical physics.

A possible application in mathematical (or quantum) physics occurs when an ordinary or a partial differential equation admits a generating function among its solutions. As a consequence, when a new generating function is constructed, it may contribute to a better understanding of the problem at hand.

Using the generalized hypergeometric function, we define
suitable functions, later denoted by  $S_1$ and $S_2$,
depending on suitable parameters. According to the values of these parameters, we find
some new relations among $S_1$, $S_2$ and the generating functions of exponential type of the coefficients of the superoscillatory function
$F_n$.

\medskip
Let $\beta\in\mathbb C$  and $v\in \mathbb N_0$. We introduce the following notation for the
Pochhammer symbol
\[
\left( \beta \right) ^{\overline{v}}=\prod\limits_{j=0}^{v-1}(\beta +j),
\]%
where $\left( \beta \right) ^{\overline{0}}=1$ for $\beta \neq 1$.
For $p,q\in\mathbb N$
$_{p}F_{q}\left[
\begin{array}{c}
\beta_{1},\beta_2,...,\beta_{p} \\
\gamma_{1},\gamma_{2},...,\gamma_{q}%
\end{array}%
;z\right] $ denotes the well-known generalized hypergeometric function,
which is defined by
\[
_{p}F_{q}\left[
\begin{array}{c}
\beta_{1},\beta_2,...,\beta_{p} \\
\gamma_{1},\gamma_{2},...,\gamma_{q}%
\end{array}%
;z\right] =\sum\limits_{m=0}^{\infty }\left( \frac{\prod\limits_{j=1}^{p}%
\left( \beta_{j}\right) ^{\overline{m}}}{\prod\limits_{j=1}^{q}\left(
\gamma_{j}\right) ^{\overline{m}}}\right) \frac{z^{m}}{m!}.
\]%
Note that the above defined series
converges for all $z$ if $p<q+1$,
and for $\left\vert
z\right\vert <1$ if $p=q+1$.

\medskip
In the above stated definition of $_{p}F_q$ all parameters $\beta_j$ can be arbitrarilyy chosen from $\mathbb{C}$.
The parameters $\gamma_{j}$ may be chosen from $\mathbb{C} \backslash -\mathbb{N}_0 = \mathbb{C} \backslash\{0,-1,-2,\cdots\}$.

Additionally, we put $%
_{0}F_{0}(z)=e^{z} $.

For more details on the function $_{p}F_{q}\left[
\begin{array}{c}
\beta_{1},\beta_2,...,\beta_{p} \\
\gamma_{1},\gamma_{2},...,\gamma_{q}%
\end{array}%
;z\right] $ see  \cite{Koepf,Lebedev,simsekJMAA,RT} and the references therein.

\begin{definition}[The functions $S_1$ and $S_2$]\label{DEFS1S2}
Let $p,q\in\mathbb N$, and let $_{p}F_{q}$ be generalized hypergeometric function.
Let $\alpha _{0},\alpha _{1},\ldots ,\alpha _{m}\in \mathbb{R}$,
 $x,t\in \mathbb{R}$ and  $k,m,n\in \mathbb{N}$.
Put
$$
\beta_{1}=\beta_2=...=\beta_{m}=k,
$$
and
$$\gamma_{1}=\gamma_{2}=...=\gamma_{m}=k+1.$$
Then we define the function:
\begin{eqnarray}
S_{1}\left( t,x;m,k,n;\alpha _{0},\alpha _{1},\ldots ,\alpha _{m}\right) &:=&%
\frac{1}{k!}\left( \frac{1-x}{2}t\right) ^{k}\sum\limits_{j=0}^{m}\alpha
_{j}\sum\limits_{l=0}^{j}\left(
\begin{array}{c}
	j  \\
	l
	
\end{array} \right)\left( -\frac{2k}{n}\right) ^{j-l} \nonumber
 \\
&&\times _{l}F_{l}\left[
\begin{array}{c}
k,k,...,k  \\
k+1,k+1,...,k+1%
\end{array}%
;\frac{1+x}{2}t\right]. \label{1Y}
\end{eqnarray}%
Put
$$
\beta_{1}=\beta_2=...=\beta_{m}=k+1,
$$
and
$$
\gamma_{1}=\gamma_{2}=...=\gamma_{m}=k.
$$
Then we define the function:
\begin{eqnarray}
S_{2}\left( t,x;m,k,n;\alpha _{0},\alpha _{1},\ldots ,\alpha _{m}\right) &:=&%
\frac{1}{k!}\left( \frac{1-x}{2}t\right) ^{k}\sum\limits_{j=0}^{m}\alpha
_{j}\sum\limits_{l=0}^{j}\left(
\begin{array}{c}
	j  \\
	l
\end{array} \right)\left( -\frac{2k}{n}\right) ^{j-l}
\nonumber \\
&&\times _{l}F_{l}\left[
\begin{array}{c}
k+1,k+1,...,k+1 \\
k,k,...,k%
\end{array}%
;\frac{1+x}{2}t\right] .  \label{1Yk}
\end{eqnarray}
\end{definition}
\begin{remark}
In Definition \ref{DEFS1S2} we have assumed that the variables $\alpha _{0},\alpha _{1},\ldots ,\alpha _{m}$
are real for our purposes.
It is clear from the definition that they can also be extended to be complex variables.
\end{remark}
\begin{remark}
Since $p=q=m$
the series converges for every $x$ and $t\in \mathbb{R}$ in view of the condition $p < q+1$ in the expression of the hypergeometric function.
\end{remark}

The functions $S_1$ and $S_2$ now generate the following sequences $b_1$ and $b_2$, respectively.

\begin{definition}[Definition of the functions $b_{1}$ and $b_{2}$]
Let $S_1$ and $S_2$ be the functions defined in Definition~\ref{DEFS1S2}. Let $x >1$, $m \in \mathbb{N}_0$, $n \in \mathbb{N}$ and $k=0,1,2,\ldots,v$. Let $\alpha_0,\ldots,\alpha_m \in \mathbb{R}$. The functions $b_1$ and $b_2$ then are implicitly defined to be their coefficient functions in the representations
\begin{equation}
S_{1}\left( t,x;m,k,n;\alpha _{0},\alpha _{1},\ldots ,\alpha _{m}\right)
=\sum_{v=0}^{\infty }b_{1}\left( v,x;m,k,n;\alpha _{0},\alpha _{1},\ldots
,\alpha _{m}\right) \frac{t^{v}}{v!},
\end{equation}
and%
\begin{equation}\label{ay-1}
S_{2}\left( t,x;m,k,n;\alpha _{0},\alpha _{1},\ldots ,\alpha _{m}\right)
=\sum_{v=0}^{\infty }b_{2}\left( v,x;m,k,n;\alpha _{0},\alpha _{1},\ldots
,\alpha _{m}\right) \frac{t^{v}}{v!}.
\end{equation}
\end{definition}

\begin{remark}
  Note that there is one generating function for each value of $k$ in the formulas (\ref{1Y}) and (\ref{1Yk}).
\end{remark}
\begin{remark}
Since the functions $S_1$ and $S_2$ defined in (\ref{1Y}) and (\ref{1Yk}) are entire in the variable $t$, we can also express $b_1$ and $b_2$ in terms of their derivatives at the origin.

\end{remark}

With the help of some special functions, special numbers and using some
integral formulas including some representations for the gamma function and the beta function, we express
some  generating functions for particular values of the parameters in  (\ref{1Y})
and (\ref{1Yk}).
To this end, we turn our attention to the case of the coefficients $c_k(n,a)$ of superoscillatory sequences $F_{n}(x,a)$.

\begin{remark}
We observe that in the representation of the coefficients $c_k(n,a)$
\begin{equation}
c_{k}(n,a)=\left(
\begin{array}{c}
	n  \\
	k
	
\end{array} \right)\left( \frac{1+a}{2}\right) ^{n-k}\left( \frac{1-a}{2}%
\right) ^{k}.  \nonumber
\end{equation}
the parameters $k$ and $n$ are such that $k=0,1,...n$, but we can give meaning for every $k$ when $n$ is fixed in terms of the Pochhammer symbol $(n)^{\bar k}$ which, for $k,n\in\mathbb N$ satisfies
\begin{equation}
\left(
\begin{array}{c}
	n  \\
	k
	
\end{array} \right)=\frac{n(n-1)....(n-k+1)}{k!}=\frac{(n)^{\bar k}}{k!}.
\end{equation}

\end{remark}
Now we introduce a generating function for the coefficients $c_{k}(n,a)$. Notice that we write $x$ instead of the parameter $a$ from now on.
\begin{definition}[Generating function of the coefficients $c_k$]
Let $c_{k}(n,x)$ be the coefficients of the superoscillating function $F_n$.  We define its
(exponential) generating function by
\begin{equation}
G_{k}(t,x):=\sum_{v=0}^{\infty }c_{k}(v,x)\frac{t^{v}}{v!}.  \label{1Y0}
\end{equation}
\end{definition}
First, let us analyze the special cases of the generating functions given in (\ref{1Y}) for $m=0,1,2$.
\begin{lemma} For every $t \in \mathbb{R}$ and $x\in \mathbb{R}$ we have
\begin{equation}\label{1Y0}
G_{k}(t,x)=\frac{1}{k!}\left( \frac{t(1-x)}{2}\right) ^{k}e^{t\left( \frac{%
x+1}{2}\right) }=\sum_{v=0}^{\infty }c_{k}(v,x)\frac{t^{v}}{v!}.
\end{equation}
\end{lemma}
\begin{proof}
Let us substitute $m=0$ into (\ref{1Y}) and (\ref{1Yk}). Then for $\alpha_0=1$
it is immediate that:%
$$
S_{1}\left( t,x;0,k,n;1\right) =S_{2}\left( t,x;0,k,n;1\right).
$$
Since
\[
G_{k}(t,x)=S_{1}\left( t,x;0,k,n;1\right),
\]%
where
\[
c_{k}(v,x)=b_{1}\left( v,x;0,k,n;1\right)
\]%
and
\[
G_{k}(t,x)=S_{2}\left( t,x;0,k,n;1\right),
\]%
where
\[
c_{k}(v,x)=b_{2}\left( v,x;0,k,n;1\right),
\]%
we deduce that
\[
G_{k}(t,x):=S_{1}\left( t,x;0,k,n;1\right) =S_{2}\left( t,x;0,k,n;1\right)=\frac{1}{k!}\left( \frac{t(1-x)}{2}\right) ^{k}e^{t\left( \frac{%
x+1}{2}\right) }
\]%
and the proof is complete.
\end{proof}

\begin{remark}
Observe that there is one generating function for each value of $k$.
\end{remark}

We now establish a relation
between $G_{k}(t,x)$ and $S_{1}\left(t,x;2,k,n;\alpha _{0},\alpha _{1}, \alpha _{2}\right)$ for arbitrary real parameters $\alpha_0,\alpha_1,\alpha_2$.

\begin{theorem}
Let $G_{k}(t,x)$ be the generating function defined in (\ref{1Y0}). Then, for arbitrary real numbers $\alpha _{0}$, $\alpha _{1}$, $\alpha _{2}$ we have the following two representations:
\begin{equation}\label{FGHJJ}
S_{1}\left( t,x;1,k,n;\alpha _{0},\alpha _{1}\right) =\left(\frac{n\alpha _{0}-2k\alpha _{1}}{n}\right) G_{k}(t,x)
\end{equation}
\begin{equation}\nonumber
+\frac{\alpha
_{1}}{k!}\left( \frac{t(1-x)}{2}\right) ^{k}\int_{0}^{1}e^{t\left( \frac{x+1%
}{2}\right) u}u^{k-1}du.  \label{1Y3}
\end{equation}

\begin{eqnarray}\label{ERT}
S_{1}\left( t,x;2,k,n;\alpha _{0},\alpha _{1}, \alpha _{2}\right)
&=&\left(\frac{n^{2}\alpha _{0}-2kn\alpha _{1}+4nk^{2}\alpha _{2}}{n^{2}}\right) G_{k}(t,x)
\nonumber
\\
&&
+\left( \frac{n\alpha _{1}-4k\alpha
	_{2} }{nk!}\right) \left( \frac{1-x}{2}t\right) \sum\limits_{v=0}^{\infty }\frac{1}{(v+k)v!}\left(
\frac{1+x}{2}t\right) ^{v}\nonumber
\\
&&
+\frac{\alpha _{2}}{k!}\left( \frac{1-x}{2}t\right)
\sum\limits_{v=0}^{\infty }\frac{1}{(v+k)^{2}v!}\left( \frac{1+x}{2}t\right)
^{v}.
\end{eqnarray}

\end{theorem}
\begin{proof}
Inserting $m=1$ in (\ref{1Y}), for every choice of $\alpha_0,\alpha_1\in \mathbb{R}$ we get
 the following relation
for the coefficients of superoscillations sequences $F_{n}(x,a)$ in (\ref%
{FNEXP}):%
\begin{eqnarray*}
S_{1}\left( t,x;1,k,n;\alpha _{0},\alpha _{1}\right)  &=&\left(\frac{n\alpha _{0}-2k\alpha _{1}}{n}\right)G_{k}(t,x)+_{1}F_{1}%
\left[
\begin{array}{c}
k \\
k+1%
\end{array}%
;\left( \frac{1+x}{2}\right) t\right]  \\
&&\times \frac{\alpha _{1}}{k!}\left( \frac{t(1-x)}{2}\right) ^{k}.
\end{eqnarray*}%
After having applied some elementary calculations, we get%
\begin{eqnarray*}
S_{1}\left( t,x;1,k,n;\alpha _{0},\alpha _{1}\right)  &=&\left(\frac{n\alpha _{0}-2k\alpha _{1}}{n}\right) G_{k}(t,x)+\frac{\alpha
_{1}}{k!}\left( \frac{t(1-x)}{2}\right) ^{k} \\
&&\times \sum_{n=0}^{\infty }\left( \frac{1+x}{2}\right) ^{n}\frac{t^{n}}{%
(n+k)n!}.
\end{eqnarray*}%
Combining the above written equation with the following integral representations of the
confluent hypergeometric function, a solution to Kummer's differential
equation (which was given by Kummer in 1837, see for detail \cite[%
Eq-(13.2.1), p. 505]{Abramowitz}):%
\[
_{1}F_{1}\left[
\begin{array}{c}
\mu  \\
\sigma
\end{array}%
;u\right] =\frac{\Gamma (\sigma -\mu )\Gamma (\mu )}{\Gamma (\sigma )}%
\int_{0}^{1}e^{uw}w^{\mu -1}\left( 1-w\right) ^{\sigma -\mu -1}dw
\]%
where $Re(\mu )>Re(\sigma )>0$, with the aid of%
\[
\Gamma (k)=(k-1)!
\]%
($k\in \mathbb N$) and%
\[
\ \Gamma (\mu +1)=\mu \Gamma (\mu),
\]%
we get (\ref{FGHJJ}).
For $m=2$, we obtain
\begin{eqnarray*}
S_{1}\left( t,x;2,k,n;\alpha _{0},\alpha _{1}, \alpha _{2}\right) &=&\left(\frac{n^{2}\alpha _{0}-2kn\alpha _{1}+4nk^{2}\alpha _{2}}{n^{2}}\right) G_{k}(t,x) \\
&&+_{1}F_{1}\left[
\begin{array}{c}
k \\
k+1%
\end{array}%
;\left( \frac{1+x}{2}\right) t\right] \left( \frac{n\alpha _{1}-4k\alpha
_{2} }{nk!}\right) \left( \frac{1-x}{2}t\right) \\
&&+_{2}F_{2}\left[
\begin{array}{c}
k,k \\
k+1,k+1%
\end{array}%
;\left( \frac{1+x}{2}\right) t\right] \frac{\alpha _{2}}{k!}\left( \frac{1-x%
}{2}t\right) .
\end{eqnarray*}%
Therefore we get the statement (\ref{ERT}).
\end{proof}

\begin{remark}
Equation (\ref{1Y3}) can be related to the following
well-known integral (in the case $u=1$):%
\[
\int_{0}^{c}e^{zx}x^{w-1}(c-x)^{u-1}dx=c^{c+w-1}B(w,u)_{1}F_{1}\left[
\begin{array}{c}
w \\
w+u%
\end{array}%
;cz\right] ,
\]%
where $Re(w)>0$ and $Re(u)>0$, $B(w,u)$ denotes the beta function (see e.g.
 \cite[Eq-3.383, p.365]{Grad})
\end{remark}

There is also an explicit connection between the Stirling numbers of the second
kind
$\mathcal{S}_{2}(c,d)$ and the function $S_{2}\left( t,x;1,m,n;\alpha _{0},\alpha _{1},\ldots ,\alpha _{m}\right)$ as we shall prove in the following theorem.

We recall that the Stirling numbers of the second
kind $\mathcal{S}_{2}(c,d)$, also known as the Stirling partition numbers, are the number of ways to partition a set of $c$ objects into $d$
non-empty subsets. As it is very well known, they have many applications in several
branches of mathematics involving probability theory, analytic
number theory (concretely in the partition theorem), generator functions, differential operators and other areas.

\begin{theorem}
Let
\[
\mathcal{S}_{2}(c,d)=\frac{1}{d!}\sum\limits_{v=0}^{d}\left(
\begin{array}{c}
	d \\
	v
	
\end{array} \right)(-1)^{v}(d-v)^{c}
\]%
be the Stirling numbers of the second
kind. Then we have
\begin{eqnarray}
S_{2}\left( t,x;1,m,n;\alpha _{0},\alpha _{1},\ldots ,\alpha _{m}\right) &=%
\frac{1}{k!}\left( \frac{1-x}{2}t\right) ^{k}\sum\limits_{j=0}^{m}\alpha
_{j}\sum\limits_{l=0}^{j}\left(
\begin{array}{c}
	j  \\
	l
	
\end{array} \right)\left( -\frac{2k}{n}\right) ^{j-l}
\nonumber \\
&\times e^{\left( \frac{1+x}{2}t\right)
}\sum\limits_{c=0}^{l}\mathcal{S}_{2}(l+1,c+1)\left( \frac{1+x}{2}t\right)^{c}. \ \
\label{1Yd}
\end{eqnarray}

\end{theorem}
\begin{proof}
Since%
\[
_{a}F_{a}\left[
\begin{array}{c}
c+1,c+1,...,c+1 \\
c,c,...,c%
\end{array}%
;z\right] =c^{-a}\sum\limits_{v=0}^{a}\left(
\begin{array}{c}
	a  \\
	v
	
\end{array} \right)c^{a-v}\sum\limits_{n=0}^{%
\infty }n^{v}\frac{z^{n}}{n!}
\]%
(\textit{cf}. \cite{Miller}), the formula (\ref{1Yk}) reduces to the
following explicit identity:%
\begin{eqnarray}
S_{2}\left( t,x;k,m,n;\alpha _{0},\alpha _{1},\ldots ,\alpha _{m}\right) &=&%
\frac{1}{k!}\left( \frac{1-x}{2}t\right) ^{k}\sum\limits_{j=0}^{m}\alpha
_{j}\sum\limits_{l=0}^{j}\left(
\begin{array}{c}
	j  \\
	l
	
\end{array} \right)\left( -\frac{2k}{n}\right) ^{j-l}
\nonumber \\
&&\times k^{-l}\sum\limits_{c=0}^{l}\left(
\begin{array}{c}
	l  \\
	c
	
\end{array} \right)k^{l-c}\sum\limits_{j=0}^{%
\infty }\frac{j^{c}}{j!}\left( \frac{1+x}{2}t\right) ^{d}.  \label{1Yc}
\end{eqnarray}

Now we use the
following identity, which can be found for example in the book by Miller and Paris \cite{Miller}:%
\begin{equation}
_{a}F_{a}\left[
\begin{array}{c}
c+1,c+1,...,c+1 \\
c,c,...,c%
\end{array}%
;z\right] =c^{-a}e^{z}\sum\limits_{v=0}^{a}\left(
\begin{array}{c}
	a  \\
	v
	
\end{array} \right)c^{a-v}\sum%
\limits_{d=0}^{v}\mathcal{S}_{2}(v,d)z^{d}.  \label{1Yb}
\end{equation}%

If we replace (\ref{1Yb}) in (\ref{1Yk}), then we get after some calculations that%
\begin{eqnarray}
S_{2}\left( t,x;k,m,n;\alpha _{0},\alpha _{1},\ldots ,\alpha _{m}\right) &=&%
\frac{1}{k!}\left( \frac{1-x}{2}t\right) ^{k}\sum\limits_{j=0}^{m}\alpha
_{j}\sum\limits_{l=0}^{j}\left(
\begin{array}{c}
	j  \\
	l
	
\end{array} \right)\left( -\frac{2k}{n}\right) ^{j-l}
\label{1Ya} \\
&&\times k^{-l}e^{\left( \frac{1+x}{2}t\right) }\sum\limits_{c=0}^{l}\left(
\begin{array}{c}
	l  \\
	c
	
\end{array} \right)k^{l-c}\sum\limits_{d=0}^{c}\mathcal{S}_{2}(c,d)\left( \frac{1+x}{2}t\right) ^{d},
\nonumber
\end{eqnarray}%
where, by definition, $0^{0}=1$.
Since
\begin{equation}\label{16a}
_{a}F_{a}\left[
\begin{array}{c}
2,2,...,2 \\
1,1,...,1%
\end{array}%
;z\right] =e^{z}\sum\limits_{v=0}^{a}\mathcal{S}_{2}(a+1,v+1)z^{v}
\end{equation}
(\textit{cf}. \cite{Miller}), the  statement is proved.
\end{proof}

\begin{remark}
Let us summarize the special cases of the generating functions given
in (\ref{1Yk}) for the values $m=0,1,2$. We have:
\begin{itemize}
\item
for $m=0$
\[
S_{2}\left( t,x;0,k,n;\alpha _{0}\right) =\alpha _{0}G_{k}(t,x);
\]
\item
for $m=1$
\begin{eqnarray*}
S_{2}\left( t,x;1,k,n;\alpha _{0},\alpha _{1}\right) &=&\left(\frac{n\alpha _{0}-2k\alpha _{1}}{n}\right) G_{k}(t,x) \\
&&+_{1}F_{1}\left[
\begin{array}{c}
k+1 \\
k%
\end{array}%
;\left( \frac{1+x}{2}\right) t\right] \frac{\alpha _{1}}{k!}\left( \frac{1-x%
}{2}t\right) ;
\end{eqnarray*}
\item
for $m=2$
\begin{eqnarray*}
S_{2}\left( t,x;2,k,n;\alpha _{0},\alpha _{1}, \alpha _{2}\right) &=&\left(\frac{n^{2}\alpha _{0}-2kn\alpha _{1}+4nk^{2}\alpha _{2}}{n^{2}}\right) G_{k}(t,x) \\
&&+_{1}F_{1}\left[
\begin{array}{c}
k+1 \\
k%
\end{array}%
;\left( \frac{1+x}{2}\right) t\right] \left( \frac{n\alpha _{1}-4k\alpha
_{2}}{nk!}\right) \left( \frac{1-x}{2}t\right) \\
&&+_{2}F_{2}\left[
\begin{array}{c}
k+1,k+1 \\
k,k%
\end{array}%
;\left( \frac{1+x}{2}\right) t\right] \frac{\alpha _{2}}{k!}\left( \frac{1-x%
}{2}t\right) .
\end{eqnarray*}
\end{itemize}\end{remark}

A relation among the function $b_{2}\left( v,x;m,k,n;\alpha _{0},\alpha
_{1},\ldots ,\alpha _{m}\right) $, the coefficients $c_k(n,a)$, and the Stirling numbers of the second kind $%
\mathcal{S}_{2}(c,d)$ is given by the following theorem.

\begin{theorem}
	Let $x>1$, $m,v\in \mathbb{N}_{0}$, and $n\in \mathbb{N}$ and $%
	k=0,1,2,\ldots ,v$. Let $\alpha _{1},\ldots ,\alpha _{m}$ be real numbers. Then
	we have%
	\begin{eqnarray}
		b_{2}\left( v,x;m,k,n;\alpha _{0},\alpha _{1},\ldots ,\alpha _{m}\right)
		&=&\sum\limits_{j=0}^{m}\alpha _{j}\sum\limits_{l=0}^{j}\left(
		\begin{array}{c}
			j \\
			l%
		\end{array}%
		\right) \left( -\frac{2k}{n}\right) ^{j-l}\sum\limits_{c=0}^{l}\left(
		\begin{array}{c}
			l \\
			c%
		\end{array}%
		\right)   \label{ay-2} \\
		&&\times \sum\limits_{d=0}^{c}\left(
		\begin{array}{c}
			v \\
			d%
		\end{array}%
		\right) \frac{d!\mathcal{S}_{2}(c,d)}{k^{c}}\left( \frac{1+x}{2}\right)
		^{d}c_{k}(v-d,x).  \nonumber
	\end{eqnarray}
\end{theorem}

\begin{proof}

Combining (\ref{ay-1}) and (\ref{1Y0}) with (\ref{1Ya}), we get%
\begin{eqnarray*}
	\sum_{v=0}^{\infty }b_{2}\left( v,x;m,k,n;\alpha _{0},\alpha _{1},\ldots
	,\alpha _{m}\right) \frac{t^{v}}{v!} &=&\sum\limits_{j=0}^{m}\alpha
	_{j}\sum\limits_{l=0}^{j}\left(
	\begin{array}{c}
		j \\
		l%
	\end{array}%
	\right) \left( -\frac{2k}{n}\right) ^{j-l}\sum\limits_{c=0}^{l}\left(
	\begin{array}{c}
		l \\
		c%
	\end{array}%
	\right)  \\
	&&\times \sum\limits_{d=0}^{c}\mathcal{S}_{2}(c,d)\left( \frac{1+x}{2}\right)
	^{d}\sum_{v=0}^{\infty }c_{k}(v,x)\frac{t^{v+d}}{k^{c}v!}.
\end{eqnarray*}%
Therefore%
\begin{eqnarray*}
	&&\sum_{v=0}^{\infty }b_{2}\left( v,x;m,k,n;\alpha _{0},\alpha _{1},\ldots
	,\alpha_{m}\right) \frac{t^{v}}{v!} =\sum\limits_{j=0}^{m}\alpha
	_{j}\sum\limits_{l=0}^{j}\left(
	\begin{array}{c}
		j \\
		l%
	\end{array}%
	\right) \left( -\frac{2k}{n}\right) ^{j-l}\sum\limits_{c=0}^{l}\left(
	\begin{array}{c}
		l \\
		c%
	\end{array}%
	\right)  \\
	&&\times \sum\limits_{d=0}^{c}\mathcal{S}_{2}(c,d)\left( \frac{1+x}{2}\right)
	^{d}\sum_{v=0}^{\infty }\left(
	\begin{array}{c}
		v \\
		d%
	\end{array}%
	\right) c_{k}(v,x)d!\frac{t^{v+d}}{k^{c}v!}.
\end{eqnarray*}%
In the previous equation, the coefficients of $\frac{t^{v}}{v!}$ of both
sides are equal. After some elementary calculations we arrive at the desired result.
\end{proof}

\begin{remark}
	Inserting $m=0$ in (\ref{ay-2}) and having in mind that $\mathcal{S}_{2}(0,0)=1$, we obtain an explicit relation between $b_2$ and the coefficients $c_k$:%
	\[
	b_{2}\left( v,x;0,k,n;\alpha _{0}\right) =\alpha _{0}c_{k}(v,x).
	\]
\end{remark}

In the special case $k=1$ we obtain the following corollary:
\begin{corollary}
	Let $t$ be a real number. Let $x>1$, $m\in \mathbb{N}_{0}$, and $n\in
	\mathbb{N}$. Let $\alpha _{1},\ldots ,\alpha _{m}$ be real numbers. We have%
	\begin{eqnarray}
		S_{2}\left( t,x;m,1,n;\alpha _{0},\alpha _{1},\ldots ,\alpha _{m}\right)
		&=&\left( \frac{1-x}{2}t\right) \sum\limits_{j=0}^{m}\alpha
		_{j}\sum\limits_{l=0}^{j}\left(
		\begin{array}{c}
			j \\
			l%
		\end{array}%
		\right) \left( -\frac{2}{n}\right) ^{j-l}  \label{1Yd} \\
		&&\times e^{\left( \frac{1+x}{2}t\right)
		}\sum\limits_{c=0}^{l}\mathcal{S}_{2}(l+1,c+1)\left( \frac{1+x}{2}t\right) ^{c}.
		\nonumber
	\end{eqnarray}
\end{corollary}
\begin{proof}
Substituting $k=1$ in (\ref{1Yk}) and using the
identity in (\ref{16a})
we get the statement.
\end{proof}

The statement of this corollary allows us to improve the statement of the previous theorem and to get a simpler formula for $b_2(v,x,m,k,n;\alpha_0,\alpha_1,\cdots,\alpha_n)$ expressed fully explicitly without the coefficients $c_k(n,x)$ rounding off this section very nicely:
\begin{theorem}
	Let $x>1$, $m,v\in \mathbb{N}_{0}$ and $n\in \mathbb{N}
	$. Let $\alpha _{1},\ldots ,\alpha _{m}$ be real numbers.
	We have%
	\begin{eqnarray*}
		b_{2}\left( v,x;m,1,n;\alpha _{0},\alpha _{1},\ldots ,\alpha _{m}\right)
		&=&\left( 1-x\right) \sum\limits_{j=0}^{m}\alpha
		_{j}\sum\limits_{l=0}^{j}\left(
		\begin{array}{c}
			j \\
			l%
		\end{array}%
		\right) \left( -\frac{2}{n}\right) ^{j-l} \\
		&&\times \sum\limits_{c=0}^{l}\left(
		\begin{array}{c}
			v \\
			c+1%
		\end{array}%
		\right) \frac{(c+1)!\mathcal{S}_{2}(l+1,c+1)\left( 1+x\right) ^{v-1}}{2^{v}}.
	\end{eqnarray*}
\end{theorem}
\begin{proof}
Combining (\ref{ay-1}) with (\ref{1Yd}), we get%
\begin{eqnarray*}
	\sum_{v=0}^{\infty }b_{2}\left( v,x;m,1,n;\alpha _{0},\alpha _{1},\ldots
	,\alpha _{m}\right) \frac{t^{v}}{v!} &=&\left( \frac{1-x}{2}\right)
	\sum\limits_{j=0}^{m}\alpha _{j}\sum\limits_{l=0}^{j}\left(
	\begin{array}{c}
		j \\
		l%
	\end{array}%
	\right) \left( -\frac{2}{n}\right) ^{j-l} \\
	&&\times \sum\limits_{c=0}^{l}\mathcal{S}_{2}(l+1,c+1)\sum_{v=0}^{\infty }\left( \frac{%
		1+x}{2}\right) ^{v+c}\frac{t^{v+c+1}}{v!}.
\end{eqnarray*}%
Therefore  %
\begin{eqnarray*}
	&&\sum_{v=0}^{\infty }b_{2}\left( v,x;m,1,n;\alpha _{0},\alpha _{1},\ldots
	,\alpha _{m}\right) \frac{t^{v}}{v!} \\
	&=&\left( \frac{1-x}{2}\right) \sum\limits_{j=0}^{m}\alpha
	_{j}\sum\limits_{l=0}^{j}\left(
	\begin{array}{c}
		j \\
		l%
	\end{array}%
	\right) \left( -\frac{2}{n}\right) ^{j-l} \\
	&&\times \sum\limits_{c=0}^{l}\mathcal{S}_{2}(l+1,c+1)\sum_{v=0}^{\infty }\left(
	\begin{array}{c}
		v \\
		c+1%
	\end{array}%
	\right) \left( \frac{1+x}{2}\right) ^{v-1}\frac{(c+1)!t^{v}}{v!}.
\end{eqnarray*}%
In the previous equation, the coefficients of $\frac{t^{v}}{v!}$ of both
sides are equal and after having applied some elementary calculations the statement follows.
\end{proof}

\section{A recurrence relation for the coefficients
$c_{k}(n,x)$ and their relation to special polynomials}\label{recurrence}

In this section we present explicit relations between the coefficients $c_k(n,x)$ of the superoscillating functions discussed previously and the Bernstein basis functions $B_{k}^{v}(x)$ as well as the Hermite polynomials $H_n(x)$. We also give some recurrence relations and relations involving the derivatives of the coefficients $c_k(n,x)$.

\subsection{Relation between $c_k(n,x)$ and the Bernstein basis functions $B_k^{v}(x)$}

We start by establishing a formula that expresses $%
c_{k}(v,x)$ in terms of the Bernstein basis functions $B_{k}^{v}(x)$, (see also \cite{SimsekAMC2011}).
Concretely,
\begin{proposition}
Let $y < 0$, $n \in \mathbb{N}$ and $k=1,2,\ldots, v$. Then we have
$$
c_k(v,1-2y)=B_k^v(y).
$$
\end{proposition}
\begin{proof}
Substituting $x=1-2y$ into (\ref{1Y0}), we obtain the well-known generating
functions for the following Bernstein basis functions:%
\[
\frac{1}{k!}\left( ty\right) ^{k}e^{t(1-y)}=\sum_{v=0}^{\infty }B_{k}^{v}(x)%
\frac{t^{v}}{v!},
\]%
where%
\[
c_{k}(v,1-2y)=B_{k}^{v}(y)=\left(
\begin{array}{c}
v  \\
	k
	
\end{array} \right)y^{k}(1-y)^{v-k}.
\]%
\end{proof}
\begin{remark}
We note that for $0\leq k\leq n$ if we substitute $x=\frac{1-y}{2}$ in the Bernstein basis functions we also get $c_k(n,y)=B_k^n\left(\frac{1-y}{2}\right)$.
\end{remark}
These Bernstein basis functions have been studied for a
long time and their applications have lead to a lot of progress in many different fields of
science.  Their generating functions have been studied extensively in recent years. For more details containing information on their properties we also refer the interested reader for example to \cite{SimsekFPTA,SimsekAcikgoz},
and also \cite{SAraci,SidamRSS,iremBS,simsekMMAS}.

\subsection{A derivative formula and a recurrence relation for $c_k(n,x)$}

In this subsection we first prove a preliminary lemma:
\begin{lemma}
Let $k,n\in \mathbb{N}$. Then we have%
\[
\frac{d}{dx}c_{k}(n,x)=\frac{n}{2}\left( c_{k}(n-1,x)-c_{k-1}(n-1,x)\right) .
\]
\end{lemma}
\begin{proof}
First of all we observe that
recalling formula (\ref{Ckna}), and
applying the same method as in \cite{SimsekFPTA} we obtain:%

\begin{eqnarray*}
\sum_{n=0}^{\infty }c_{k}(n,x)\frac{t^{n}}{n!} &=&\frac{1}{k!}\left( \frac{%
1-x}{2}\right) ^{k}\sum_{n=0}^{\infty }\left( \frac{x+1}{2}\right) ^{n}\frac{%
t^{n+k}}{n!} \\
&=&\frac{1}{k!}\left( t\frac{1-x}{2}\right) ^{k}\sum_{n=k}^{\infty }\left(
\frac{x+1}{2}\right) ^{n-k}\frac{t^{n}}{(n-k)!} \\
&=&\frac{1}{k!}\left( t\frac{1-x}{2}\right) ^{k}e^{\frac{x+1}{2}%
t}=G_{k}(t,x).
\end{eqnarray*}

By taking the derivative with some computations we get:
\[
\frac{\partial}{\partial {x}}G_{k}(t,x)=\frac{t}{2}\left( G_{k}(t,x)-G_{k-1}(t,x)\right),
\]%
and this equality can be used to obtain the derivative formula of the polynomials $c_{k}(n,x)$. Indeed, combining the above given relation with (\ref{1Y0}), we get%
\[
\sum_{n=0}^{\infty }\frac{d}{dx}c_{k}(n,x)\frac{t^{n}}{n!}=\frac{1}{2}%
\sum_{n=0}^{\infty }c_{k}(n,x)\frac{t^{n+1}}{n!}-\frac{1}{2}%
\sum_{n=0}^{\infty }c_{k-1}(n,x)\frac{t^{n+1}}{n!}.
\]%
Comparing the coefficients of $\frac{t^{n}}{n!}$ on both sides of the above
equation, we obtain the statement.
\end{proof}

This tool in hand permits us to deduce the following recurrence relation for the polynomials $c_{k}(n,x)$:

\begin{theorem}
	Let $x>1$. Let $n\in \mathbb{N}_{0}$ and $k=1,2,\ldots ,n$. We have%
	\[
	c_{k}(n+1,x)=\frac{1-x}{2}c_{k-1}(n,x)+\frac{1+x}{2}c_{k}(n,x).
	\]
\end{theorem}

\begin{proof}
By combining the following equation%
\[
\frac{\partial}{\partial {t}} G_{k}(t,x)=\frac{1-x}{2}G_{k-1}(t,x)+\frac{1+x}{2}G_{k}(t,x)
\]%
considering the partial derivation with respect to $t$ with equation (\ref{1Y0}), we obtain a recurrence relation for the polynomials $c_{k}(n,x)$.%
\[
\sum_{n=0}^{\infty }c_{k}(n+1,x)\frac{t^{n}}{n!}=\frac{1-x}{2}%
\sum_{n=0}^{\infty }c_{k-1}(n,x)\frac{t^{n}}{n!}+\frac{1+x}{2}%
\sum_{n=0}^{\infty }c_{k}(n,x)\frac{t^{n}}{n!}.
\]%
A comparison of the coefficients of $\frac{t^{n}}{n!}$ on both sides of the above mentioned
equation leads to the statement.
\end{proof}

\subsection{Relation between the Hermite polynomials and the polynomials $%
	c_{k}(n,x)$}

In this subsection, we also establish relations between the Hermite polynomials and the
polynomials $c_{k}(n,x).$ In order to deduce these relations, we need the
following generating functions for the generalized Hermite polynomials called
Gould-Hopper polynomials $H_{n}^{(j) }(x,y) $. The latter are defined by means of the
following generating function: Let $j \in \mathbb{N}$. The generating
function for the polynomials $H_{n}^{(j) }(x,y) $ is given by
\begin{equation}
	G_{H}\left( t,x,y,j\right) =\exp \left( xt+yt^{j}\right)
	=\sum\limits_{n=0}^{\infty }H_{n}^{\left( j\right) }\left( x,y\right) \frac{
		t^{n}}{n!},  \label{H1}
\end{equation}%
(\cite{Ceseno,gould,Lebedev}). The polynomials $%
H_{n}^{\left( 2\right) }\left( x,y\right) $ satisfy the following heat
equation:%
\[
\frac{\partial }{\partial y}\left( H^{(2)}(x,y)\right) =\frac{\partial ^{2}%
}{\partial x^{2}}\left( H^{(2)}(x,y)\right)
\]
(\cite{Ceseno}).

Substituting $y=-1$, $j=2$ and $x=2z$ into (\ref{H1}), we obtain the following
generating function for the well-known Hermite polynomials, $%
H_{n}(z):=H_{n}^{\left( 2\right) }\left( 2z,-1\right)$:
\[
F_{H}(t,z)=e^{2zt-t^{2}}=\sum_{n=0}^{\infty }H_{n}(z)\frac{t^{n}}{n!}.
\]%
From the previous equation, one may infer that
\[
H_{m}(z)=(-1)^{m}e^{z^{2}}\frac{d^{m}}{dz^{m}}\left( e^{-z^{2}}\right) ,
\]%
$m\in \mathbb{N}_{0}$ (\textit{cf}. \cite{Lebedev}), and also
\[
H_{n}(z) =n!\sum\limits_{s=0}^{\left[ \frac{n}{2}\right] }(-1)^{s}\frac{%
	\left( 2x\right) ^{n-2s}}{\left( n-2s\right) !s!}
\]

where $\left[ x\right] $ denote the largest integer $\leq x$ (see
\cite{Ceseno,gould,Lebedev}).
Using (\ref{H1}), one can easily derive the above stated partial differential
equation. The Hermite polynomials are related to
Kummer's confluent hypergeometric functions:
\[
H_{2m}(z)=(-1)^{m}\frac{(2m)!}{m!}\left(_{1}F_{1}\left[
\begin{array}{c}
	-m \\
	\frac{1}{2}%
\end{array}%
;z^{2}\right] \right)
\]%
and%
\[
H_{2m+1}(z)=(-1)^{m}\frac{2(2m+1)!z}{m!} \left(_{1}F_{1}\left[
\begin{array}{c}
	-m \\
	\frac{3}{2}%
\end{array}%
;z^{2}\right] \right)
\]%
(\textit{cf}. \cite{Lebedev}).

Hermite polynomials have many important applications. Let us give a short overview.
for instance, they constitute important orthogonal polynomial sequences, they are used in signal processing as
Hermitian wavelets associated with the wavelet transform. Furthermore they play a crucial role in probability theory for instance in connection with the
Brownian motion as Edgeworth series, or as probability distribution functions such as the Gaussian distribution and in combinatorics. Moreover this, they appear in Appell sequence applications in the umbral calculus and in numerical analysis as Gaussian
quadrature. In physics they reveal the eigenstates of the quantum
harmonic oscillator and also occur as solutions in the heat equation.

We now consider the following functional equation:%
\begin{equation}
	G_{k}(t,x)e^{-t^{2}}=\frac{1}{k!}\left( t\frac{1-x}{2}\right)
	^{k}F_{H}\left( t,\frac{1+x}{4}\right).  \label{1Yh}
\end{equation}%
This functional equation allows us to establish particularly a relation between the Hermite
polynomials $H_{n}(x,y)$ and the polynomials $c_{k}(n,x)$.

\begin{theorem}
	Let $x>1$. Let $n\in \mathbb{N}_{0}$ and $k=0,1,2,\ldots ,n$. We have%
	\[
	\left( \frac{1-x}{2}\right) ^{k}H_{n-k}\left( \frac{1+x}{4}\right) =\frac{n!%
	}{\left(
		\begin{array}{c}
			n \\
			k%
		\end{array}%
		\right) }\sum_{j=0}^{[\frac{n}{2}]}\frac{\left( -1\right) ^{j}c_{k}(n-2j,x)}{%
		j!(n-2j)!}.
	\]
	where $[x]$ denotes the largest integer $\leq x$.
\end{theorem}

\begin{proof}
	Using (\ref{1Yh}), we have

\[
\sum_{n=0}^{\infty }c_{k}(n,x)\frac{t^{n}}{n!}\sum_{n=0}^{\infty }(-1)^{n}%
\frac{t^{2n}}{n!}=\frac{1}{k!}\left( \frac{1-x}{2}\right)
^{k}\sum_{n=0}^{\infty }H_{n}\left( \frac{1+x}{4}\right) \frac{t^{n+k}}{n!}.
\]%
Therefore%
\[
\left( \frac{1-x}{2}\right) ^{k}\sum_{n=0}^{\infty } \left(
\begin{array}{c}
	n \\
	k%
\end{array}
\right) H_{n-k}\left( \frac{1+x}{4}\right) \frac{t^{n}}{n!}%
=\sum_{n=0}^{\infty }\sum_{j=0}^{[\frac{n}{2}]}\frac{\left( -1\right)
	^{j}c_{k}(n-2j,x)}{ j!(n-2j)!}t^{n}.
\]%
After some elementary calculations, comparing the coefficients of $t^{n}$ on
both sides of the previous equation, we arrive at the desired result.

\end{proof}

\section{Discussion and concluding remarks}\label{CONC}

The previous sections show some deep connections among the coefficients of the the superoscillatory functions and several special functions.
Also different functions of superoscillatory type can be considered.
 For example one can consider the derivatives
of the function $F_{n}(x,a)$ given by
\[
\frac{\partial^{p}}{\partial {x}^{p}}F_{n}(x,a)=\sum_{j=0}^{n}c_{j}(n,a)(i(1-2j/n))^{p}e^{i(1-2j/n)x},
\]%
where $x\in \mathbb{R}$, and computing the limit we have
\[
\lim_{n\rightarrow \infty }\frac{\partial^{p}}{\partial {x}^{p}}F_{n}(x,a)=(ia)^{p}e^{iax}
\]%
uniformly on the compact sets.
So sequences of the type $\dfrac{\partial^{p}}{\partial {x}^{p}}F_{n}(x,a)$
will constitute a class of generalized superoscillating sequence (often called supershifts). It is clear
that we can consider more general functions such as
\[
\frac{\partial^{p}}{\partial {x}^{p}} Z_{n}(x,a):=%
\sum_{j=0}^{n}c_{j}(n,a)(ik_{j}(n))^{mp}e^{i(k_{j}(n))^{m}x},
\]%
with
$
k_{j}(n)=1-2j/n.
$
It has been proved that under suitable conditions we have
\[
\lim_{n\rightarrow \infty }\frac{\partial^{p}}{\partial {x}^{p}} Z_{n}(x,a)=(ia)^{mp}e^{ia^{m}x}
\]

The natural spaces where we study these problems are the spaces of entire functions with growth conditions that are defined as follows.
Let $p\geq 1$ and denote by $\mathcal{A}_{p}$ the space of entire functions
with either order lower than $p$ or order equal to $p$ and finite type. In the space $\mathcal{A}_{p}$
consists of functions $f$ for which there exist constants $B,C>0$ such that
\begin{equation}
|f(z)|\leq Ce^{B|z|^{p}}.  \label{ABC}
\end{equation}%
Let $(f_{n})_{n\in \mathbb{N}}$, $f_{0}\in \mathcal{A}_{p}$. Then $%
f_{n}\rightarrow f_{0}$ in $\mathcal{A}_{p}$ if there exists some $B>0$ such
that
\begin{equation}
\lim\limits_{n\rightarrow \infty }\sup_{z\in \mathbb{C}}\left\vert
(f_{n}(z)-f_{0}(z))e^{-B|z|^{p}}\right\vert =0.
\end{equation}

The class of functions that we will further investigate are of the following type.
\begin{theorem}
Let us consider
\begin{equation}
\mathcal{Y}_{n}(x,a)=\sum_{j=0}^{n}c_{j}(n,a)h(ik_{j}(n))e^{ig(k_{j}(n))x}
\label{FNEXP-1}
\end{equation}%
where $c_{k}(n,a)$ are given by (\ref{Ckna}), $k_{j}(n)=1-2j/n$ and $g$ and $%
h$ are entire functions of the type $A_{p}$. Then we have
\[
\lim_{n\rightarrow \infty }\mathcal{Y}_{n}(x,a)=h(a)e^{ig(a)x}.
\]%
In particular if $g(a)\leq 1$ for $|a|\leq 1$ and $g(a)>1$ for $|a|>1$ then $%
\mathcal{Y}_{n}(x,a)$ is a generalized superoscillatory function.
\end{theorem}

 This theorem is a particular case of more general results existing in the literature.
 We just want to mention that its proof uses sophisticated tool such as infinite order differential operators and their continuity in the spaces of entire functions.
 The main points of the proofs consist in considering the series expansion of entire functions
\[
g(\lambda )=\sum_{u=0}^{\infty }g_{u}\lambda ^{u}
\ \ \ \ \ \ h(\lambda )=\sum_{u=0}^{\infty }h_{u}\lambda ^{u},
\]%
where $\lambda \in \mathbb{C}$. We observe that
\[
e^{ixg(\lambda )}=\prod_{u=0}^{\infty }\sum_{m=0}^{\infty }\frac{%
(ixg_{u})^{m}}{m!}\lambda ^{um}.
\]%
Moreover, from the above stated expression we get
\[
h(\lambda )e^{ixg(\lambda )}=\sum_{v=0}^{\infty }h_{v}\lambda
^{v}\prod_{u=0}^{\infty }\sum_{m=0}^{\infty }\frac{(ixg_{u})^{m}}{m!}\lambda
^{um}
\]%
and also
\[
h(\lambda )e^{ixg(\lambda )}=\sum_{v=0}^{\infty }h_{v}\prod_{u=0}^{\infty
}\sum_{m=0}^{\infty }\frac{(ixg_{u})^{m}}{m!}\lambda ^{um+v}.
\]%
Now define the operator
\[
\mathcal{U}_{h,g}(x,D_\xi):=\sum_{v=0}^{\infty }h_{v}\prod_{u=0}^{\infty
}\sum_{m=0}^{\infty }\frac{(ixg_{u})^{m}}{m!}i^{-(um+v)}D_{\xi }^{um+v}
\]%
where $D_\xi=\frac{\partial }{\partial \xi}$, $\xi$ is an auxiliary complex variable
and observe that we can write:
\[
h(a)e^{ixg(a)}=\mathcal{U}_{h,g}(x,D_{\xi })e^{ia\xi }|_{\xi =0}.
\]%
The delicate point is that we can prove
the continuity of the operator
 $$\mathcal{U}_{h,g}(x,D_{\xi}):A_{1}\rightarrow A_{1}.
$$ So we can write
\[
\mathcal{Y}_{n}(x,a;\xi )=\mathcal{U}_{h,g}(x,D_{\xi })F_{n}(\xi ,a).
\]%
In view of of the continuity we have
\[
\lim_{n\rightarrow \infty }\mathcal{Y}_{n}(x,a)=\mathcal{U}_{h,g}(x,D_{\xi
})\lim_{n\rightarrow \infty }F_{n}(\xi ,a)|_{\xi =0}=\mathcal{U}%
_{h,g}(x,D_{\xi })e^{iax}|_{\xi =0}.
\]%
where $|_{\xi =0}$ denotes the restriction, so we finally get
\[
\lim_{n\rightarrow \infty }\mathcal{Y}_{n}(x,a)=\mathcal{U}_{h,g}(x,D_{\xi
})e^{ia\xi }|_{\xi =0}=h(a)e^{ig(a)x}.
\]

\noindent{\bf Data availability statement}.
There are no data in support the findings of this study.

\medskip\noindent
{\bf Conflict of interest}.
On behalf of all authors, the corresponding author states that there is no conflict of interest.

\medskip\noindent
{\bf OrcIDs}. Fabrizio Colombo	https://orcid.org/0000-0002-7066-8378

\end{document}